\documentclass[11pt,twoside]{amsart}

\usepackage{latexsym}
\usepackage{amsmath}
\usepackage{amssymb}
\usepackage{amsthm}
\usepackage{amscd}
\usepackage[all]{xy}
\usepackage{enumerate} 
\usepackage{amssymb} 
\usepackage{mathrsfs}

\def\ZZ{{\mathbb{Z}}}% \ZZ == \mathbb{Z}
\def\QQ{{\mathbb{Q}}}% \QQ == \mathbb{Q}
\def\RR{{\mathbb{R}}}% \RR == \mathbb{R}
\def\C{{\mathscr{C}}}% \C == \mathscr{C}
\def\CC{{\mathbb{C}}}% \CC == \mathbb{C}

\def\K{{\mathscr{K}}}

\def\PP{{\mathbb{P}}}% \PP == \mathbb{P}
\def\E{{\mathscr{E}}}% \E == \mathscr{E}
\def\H{{\mathscr{H}}}% \H == \mathscr{H}
\def\M{{\mathscr{M}}}% \M == \mathscr{M}
\def\V{{\mathscr{V}}}% \V == \mathscr{V}
\usepackage{latexsym}

\newtheorem{them}{Theorem}[section]

\newtheorem{pro}[them]{Proposition}

\newtheorem{lem}[them]{Lemma}

\newtheorem{rem}[them]{Remark}

\newtheorem{cor}[them]{Corollary}

\newtheorem{nota}[them]{Notation}

\newtheorem{cl}[them]{Claim}

\begin{document}

\title[On projective manifolds swept out by cubic varieties]{On projective manifolds swept out by cubic varieties}
\author{Kiwamu Watanabe}
\date{\today}

\address{Graduate School of Mathematical Sciences, University of Tokyo, 
3-8-1 Komaba Meguro-ku Tokyo 153-8914, Japan.}
\email{watanabe@ms.u-tokyo.ac.jp}


\maketitle

%%%%%%%%%%%%%%%%%%%%%%%%%%%%%%%%%%%%%%%%%%%%%%%%%%%%%%%%%%%%%

\begin{abstract}
We study structures of embedded projective manifolds swept out by cubic varieties. We show if an embedded projective manifold is swept out by high-dimensional smooth cubic hypersurfaces, then it admits an extremal contraction which is a linear projective bundle or a cubic fibration. As an application, we give a characterization of smooth cubic hypersurfaces. We also classify embedded projective manifolds of dimension at most five swept out by copies of the Segre threefold $\PP^1 \times \PP^2$.  In the course of the proof, we classify projective manifolds of dimension five swept out by planes.     
\end{abstract}

\tableofcontents

\section{Introduction}

The structures of embedded complex projective manifolds swept out by varieties with preassigned properties have been studied by several authors. For example, E. Sato showed that projective manifolds covered by linear subspaces of dimension at least $[\frac{n}{2}]+1$ are linear projective bundles \cite{S}. Recently, M. C. Beltrametti and P. Ionescu proved that projective manifolds swept out by quadrics of dimension at least $[\frac{n}{2}]+2$ are either linear projective bundles or quadric fibrations \cite{BI} (see also \cite{KS} due to Y. Kachi and Sato). Their result recovers B. Fu's one on a characterization of quadrics \cite{Fu}. Remark that related problems were also dealt in \cite{MS}, \cite{Wat}. 

In this paper, we investigate structures of embedded projective manifolds swept out by cubic varieties. 
Recall that an embedded projective manifold of degree $3$ is a cubic hypersurface, the Segre $3$-fold $\PP^1 \times \PP^2 \subset \PP^5$ or one of its linear sections, that is, either a cubic scroll or a twisted cubic curve \cite{X}. By a slight abuse of terminology, we say that an embedded projective manifold is a {\it hypersurface}     
 if it has codimension $1$ in its linear span. 
We prove the following:

\begin{them}\label{hyp} Let $X \subset \PP^N$ be a complex projective manifold of dimension $n \geq 5$. Assume that there is a smooth cubic hypersurface in $X$ of dimension $s > [\frac{n}{2}]+2$ through a general point of $X$. Then $X$ admits a contraction of an extremal ray $\varphi  :X \rightarrow Y$ which contracts the cubic hypersurfaces and it is a linear projective bundle or a cubic fibration (, which means a fibration whose general fibers are cubic hypersurfaces). 
\end{them} 

We immediately obtain a characterization of smooth cubic hypersurfaces as a corollary. This is a cubic version of the theorems on characterizations of quadrics due to Kachi-Sato \cite[Main Theorem~0.2]{KS}, Fu \cite[Theorem~2]{Fu} and Beltrametti-Ionescu \cite[Corollary~3.2]{BI}:

\begin{cor}\label{characubic} Let $X \subset \PP^N$ be a nondegenerate complex smooth projective $n$-fold with $5 \leq n <N$. Then the following are equivalent:
\begin{enumerate}
\item $X$ is a cubic hypersurface.
\item Through a general point $x \in X$, there is a smooth cubic hypersurface $S_x \subset X$ of dimension $s > [\frac{n}{2}]+2$ and the cubics $S_x$ and $S_{x'}$ intersect for two general points $x,x' \in X$.
\end{enumerate}
\end{cor} 

We also classify embedded projective manifolds of dimension at most $5$ swept out by copies of the Segre $3$-fold $\PP^1 \times \PP^2$:

\begin{them}\label{Segre3} Let $X \subset \PP^N$ be a complex projective manifold of dimension $n \leq 5$. Assume that there is a copy of the Segre $3$-fold $\PP^1 \times \PP^2$ through a general point of $X$ in $X$. Then $X$ is one of the following:
\begin{enumerate}
\item a linear $\PP^d$-bundle, where $d \geq 2$,
\item a hyperplane section of the Grassmann variety $G(2,\CC^5) \subset \PP_{\ast}(\bigwedge^2\CC^5)$, 
\item $\PP^1 \times Q^4 \subset \PP^{11}$, or
\item $X$ admits a contraction of an extremal ray $\varphi  :X \rightarrow C$ to a smooth curve whose general fiber is projectively equivalent to a Segre $4$-fold $\PP^2 \times \PP^2 \subset \PP^8$.
\end{enumerate}
\end{them}

Remark that there are many examples of projective $n$-folds swept out by copies of the Segre $3$-fold $\PP^1 \times \PP^2$ in the case where $n>5$, e.g., manifolds swept out by $5$-dimensional linear subspaces. From this viewpoint, the assumption that $n \leq 5$ seems natural. 

The strategy to show Theorem~\ref{characubic} and Theorem~\ref{Segre3} is as follows: The manifold $X$ as in Theorem~\ref{characubic} and Theorem~\ref{Segre3} is covered by lines. We produce a large covering family of lines based on deformation theory of rational curves. It turns out that the family defines a contraction of an extremal ray. Then a general fiber is a Fano manifold of coindex at most $3$. Investigating the fiber by using classification results of Fano manifolds, we classify $X$ as in Theorem~\ref{characubic} and Theorem~\ref{Segre3}. 

The contents of this paper are organized as follows. In Section~2, we recall some facts on deformation theory of rational curves and explain results on extremal contractions and Fano manifolds. In particular, we review two results due to C. Novelli and G. Occhetta (Theorem~\ref{NO} and Theorem~\ref{NO2}) which play an important role in our proof. It asserts that a large covering family of lines spans an extremal ray of the Kleiman-Mori cone. Section~3 is devoted to prove Theorem~\ref{hyp}. The idea of our proof is analogous to one of \cite{BI}. One of significant steps is to show the existence of a covering family of lines induced from cubic hypersurfaces (Lemma~\ref{famcubic}). Section~4 deals with a classification of projective $5$-folds swept out by planes (Theorem~\ref{plane}). In \cite[Main Theorem]{S}, Sato classified projective manifolds swept out by linear subspaces of dimension at least $[\frac{n}{2}]$. In this sense, our case is a next case to the one in \cite[Main Theorem]{S}. Although a related topic was dealt in \cite[Theorem~1.1, Question~6.2]{NO}, our result does not follow from theirs. Actually, we give a negative answer to \cite[Question~6.2]{NO} in Remark~\ref{NOQ}. More precisely, we give an example of an embedded projective manifold of dimension $2n+1$ which is swept out by linear subspaces of dimension $n$ but does not contain a linear subspace of dimension $n$ with nef normal bundle. In Section~5, we prove Theorem~\ref{Segre3}. The main difficulty lies in the case where the dimension of $X$ is $5$. Since projective manifolds swept out by copies of the Segre $3$-fold $\PP^1 \times \PP^2$ are also swept out by planes, we can narrow the possibilities of manifolds by using Theorem~\ref{plane}. Furthermore we also investigate a contraction of an extremal ray which is defined by a covering family of lines induced from lines on the first factors of copies of the Segre $3$-fold $\PP^1 \times \PP^2$. 

In this paper, we work over the field of complex numbers. \\  

{\bf Acknowledgements.} 
The author would like to express his deep gratitude to Professor Hajime Kaji for his valuable suggestions and advice. The author is also greatly indebted to Professor Hiromichi Takagi for reading a preliminary version of this paper and his several helpful comments. In particular, he taught the author a part of the proof of Proposition~\ref{P2} (Remark~\ref{takagi}). The author would also like to thank the referee for careful reading and useful comments which simplified proofs of Proposition~\ref{sch} and \ref{fb}.
He would like to thank Dr. Katsuhisa Furukawa, Dr. Yoshinori Gongyo, Dr. Shinpei Ishii and Dr. Daizo Ishikawa for their encouragement during my studies. He is partially supported by Research Fellowships of the Japan Society for the Promotion of Science for Young Scientists.

\section{Preliminaries}

We use the notation and terminology as in \cite{BS}, \cite{H}, \cite{Hw2}, \cite{Ko}. For a vector space $V$, $\PP_{{\ast}}(V)$ denotes the projective space of lines through the origin in $V$. 

\subsection{Families of lines and varieties of minimal rational tangents}\label{l}

Let $X \subset \PP^N$ be a smooth projective variety of dimension $n$, which is called a {\it projective manifold of dimension $n$} or a {\it smooth projective $n$-fold}. For an $m$-tuple of numerical polynomials $\PP(t):=(P_1(t),\cdots,P_m(t))$, we denote by ${\rm FH}_{\PP(t)}(X)$ the flag Hilbert scheme relative to $\PP(t)$, which parametrizes $m$-tuples $(Z_1, \cdots, Z_m)$ of closed subschemes of $X$ such that $Z_i$ has Hilbert polynomial $P_i(t)$ and $Z_1 \subset \cdots \subset Z_m$ (see \cite[Section~4.5]{Ser}). 
In particular, for a numerical polynomial $P(t) \in \QQ[t]$, we denote by ${\rm Hilb}_{P(t)}(X)$ the Hilbert scheme relative to $P(t)$. We also denote by $F_1(X)$ the Hilbert scheme of lines on $X$. An irreducible component of $F_1(X)$ is called a {\it family of lines} on $X$. For a point $x \in X$, we denote by $F_1(x,X)$ the Hilbert scheme of lines in $X$ passing through $x$. 
Let ${\rm Univ}(X)$ be the universal family of ${\rm Hilb}(X)$ with the associated morphisms $\pi :{\rm Univ}(X) \rightarrow {\rm Hilb}(X)$ and $\iota : {\rm Univ}(X) \rightarrow X$. 
For a subset $V$ of ${\rm Hilb}(X)$, $\iota ({\pi}^{-1}(V))$ is denoted by ${\rm Locus}(V) \subset X$. 
A {\it covering family of lines} $\K$ means an irreducible component of $F_1(X)$ satisfying ${\rm Locus}(\K)=X$. For a family of lines $\K$, since all members of $\K$ are numerically equivalent, we can define the intersection number of $\K$ with a divisor $D$ on $X$ as the one of $D$ and a member of $\K$ and we denote it by $D.\K$. For a subset $\K \subset F_1(X)$, a {\it $\K$-line} means a line which is a member of $\K$.

From here on, assume that $X \subset \PP^N$ is covered by lines. A line $l$ is {\it free} if the restriction $T_X|_l$ of the tangent bundle of $X$ to $l$ is nef. If a line $l$ is free, $F_1(X)$ is smooth at $[l]$ (see \cite[Theorem~4.3.5]{Ser}). In particular, $F_1(X)$ has a unique irreducible component containing $[l]$. A family of lines $\K$ is a covering family if and only if $\K$ contains a free line (see \cite[II. Corollary~3.5.3 and Proposition~3.10]{Ko}). 
For a covering family of lines $\K$ on $X$, there exists a minimal proper closed subset $Z_{\K} \subset X$ such that every line through a point of $X \setminus Z_{\K}$ is free. 
In fact, since being free is an open condition and $\K$ is proper, the locus of non-free lines in $\K$ is closed and it is exactly $Z_{\K}$. 
 We call $Z_{\K}$ the {\it non-free locus of $\K$}.  
Let $x$ be a point on $X \setminus Z_{\K}$ and set $\K_x:=\{[l] \in \K|x \in l\}$.
For any $\K_x$-line $l$, the tangent bundle $T_X$ satisfies $T_X|_l \cong {\mathscr{O}}_{\PP^1}(2) \oplus {\mathscr{O}}_{\PP^1}(1)^{p} \oplus {\mathscr{O}}_{\PP^1}^{n-p-1}$ for some non-negative integer $p$. From a fundamental argument of deformation theory of rational curves, $\K_x$ is a disjoint union of smooth varieties of dimension $p$ (see \cite[Theorem~1.3]{Hw2}). We define the tangent map ${\tau}_x : F_1(x,X) \rightarrow \PP_{{\ast}}(T_xX)$ by assigning the tangent vector at $x$ to each member of $F_1(x,X)$. 
It is known that the tangent map ${\tau}_x$ is an embedding (see \cite[Proposition~1.5]{Hw2}). Hence $F_1(x,X)$ can be seen as a (possibly disconnected) subvariety of $\PP_{{\ast}}(T_xX)$. We denote by $\C_x \subset \PP_{{\ast}}(T_xX)$ the image ${\tau}_x(\K_x)$, which is called the {\it variety of minimal rational tangents} (with respect to $\K$) at $x$. The following proposition is a fundamental result on varieties of minimal rational tangents.

\begin{pro}[{\cite[Proposition~1.5, Theorem~2.5, Section~5~Question~2]{Hw2}}]\label{VMRT} Let $X \subset \PP^N$ be a smooth projective $n$-fold covered by lines, $\K$ a covering family of lines on $X$ and $\C_x$ its variety of minimal rational tangents at a general point $x \in X$. Assume that ${\rm Pic}(X) \cong \ZZ[{\mathscr{O}}_X(1)]$. If $\dim \C_x \geq \frac{n-1}{2}$, then $\C_x \subset \PP_{{\ast}}(T_xX)$ is smooth, irreducible and nondegenerate.
\end{pro}

\begin{pro}\label{fano} Let $X \subset \PP^N$ be a projective manifold covered by lines and $x$ a general point on $X$. Assume that $F_1(x,X)$ is irreducible. Then there exists a unique covering family of lines. In particular, its variety of minimal rational tangents at $x$ coincides with $F_1(x,X) \subset \PP_{{\ast}}(T_xX)$. 

\end{pro}

\begin{proof} Assume the contrary. Let 
\begin{eqnarray}
F_1(X)= \bigcup F^i \cup \bigcup G^j \nonumber
\end{eqnarray}
be the irredundant irreducible decomposition, where $F^i$'s are covering families and $G^j$'s are not. From our assumption, we have at least two $F^i$'s. 
For a general point $x$ on $X$, 
\begin{eqnarray}
F_1(x,X)= \bigcup F^i(x,X), \nonumber
\end{eqnarray}
where $F^i(x,X)$ is the family of lines from $F^i$ passing through $x$.
Then we see that $F^i(x,X)$ is non-empty and $F^i(x,X)$ is not contained in $F^{i'}$ for $i \neq i'$. This is a contradiction. Consequently, $X$ has a unique covering family of lines.
\end{proof}

\begin{pro}\label{nf} Let $X \subset \PP^N$ be a projective manifold and $Y$ a smooth hyperplane section of $X$. Assume that $Y$ is covered by lines. Then the following holds.
\begin{enumerate}
\item There exists a covering family of lines $\K_X$ for $X$ which contains a covering family of lines for $Y$. In particular, $X$ is also covered by lines.
\item $Y$ is not contained in the non-free locus $Z_{\K_X}$.
\end{enumerate}   

\end{pro}

\begin{proof} $\rm (i)$ Since $Y$ is covered by lines, there exists a covering family of lines $\K_Y$ for $Y$ and a free $\K_Y$-line $l$ in $Y$. Then we have an exact sequence
\begin{eqnarray}
0 \rightarrow T_Y|_l \rightarrow T_X|_l \rightarrow N_{Y/X}|_l \rightarrow 0. \nonumber
\end{eqnarray}
We have $T_Y|_l \cong {\mathscr{O}}_{\PP^1}(2) \oplus {\mathscr{O}}_{\PP^1}(1)^{p} \oplus {\mathscr{O}}_{\PP^1}^{q}$ for some non-negative integers $p, q$ and $N_{Y/X}|_l \cong {\mathscr{O}}_{\PP^1}(1)$. Since this sequence splits, $l$ is also free in $X$. Then there is a unique family of lines $\K_X$ for $X$ containing $[l]$. Hence $\K_X$ is a covering family of lines which contains $\K_Y$.\\ 
$\rm (ii)$ Let 
\begin{eqnarray}\label{Z} 
Z_{\K_X}=Z_1 \cup \cdots \cup Z_m \nonumber
\end{eqnarray}
be the irredundant irreducible decomposition of the non-free locus of $\K_X$.  
Assume that $Y$ is contained in $Z_{\K_X}$. Then $Y$ is contained in $Z_{i_0}$ for some $i_0$. Since we have ${\rm codim}_XY=1$, $Y$ coincides with $Z_{i_0}$.  
From the definition of the non-free locus, through a point $y$ on $Y$, there is a $\K_X$-line $l_y$ which is not free in $X$. 
Then $l_y$ is not contained in $Y$ for a general point $y$ on $Y$. In fact, if $Y$ contains $l_y$, then it follows from the generality of $y$ that $l_y$ is free in $Y$. The same argument as in $\rm (i)$ implies that $l_y$ is also free in $X$. This is a contradiction. 
Hence one sees that $l_y$ is contained in $Z_{i_1}$ for some $i_1 \neq i_0$, provided that $y$ is general. In particular, $Z_{i_1}$ contains a general point $y$. This implies that $Y$ is contained in 
\begin{eqnarray}
Z_1 \cup \cdots \cup \hat{Z_{i_0}} \cup \cdots \cup Z_m. \nonumber
\end{eqnarray}
Therefore $Z_{i_0}$ coincides with $Z_{i_2}$ for some $i_2 \neq i_0$. This gives rise to a contradiction.
\end{proof}

\begin{pro}\label{sechomog} Let $W$ be a rational homogeneous manifold of Picard number $1$ and $W \subset \PP^N$ the embedding defined by the very ample generator 
 of the Picard group of $W$. Let $\Lambda \subset \PP^N$ be a linear subspace, $X$ the linear section of $W$ by $\Lambda$ and $x$ a general point on $X$. Assume that $X$ is a projective manifold covered by lines and it is not a projective space. If $\dim F_1(x,W)>{\rm codim}_{\PP^N}\Lambda$, then $F_1(x,X)$ is a smooth and irreducible variety of degree $>1$. Furthermore there exists a unique covering family of lines for $X$ and $F_1(x,X) \subset \PP_{{\ast}}(T_xX)$ is its variety of minimal rational tangents at $x$.  
\end{pro}

\begin{proof} For a general point $x \in X$, $F_1(x,X)$, $F_1(x,W)$ and $F_1(x,\Lambda)$ can be naturally embedded into $\PP_{{\ast}}(T_x\PP^N)$. Then $F_1(x,X)=F_1(x,W) \cap F_1(x,\Lambda)$ in $\PP_{{\ast}}(T_x\PP^N)$. It is known that $F_1(x,W)$ is irreducible variety of degree $>1$ (see \cite{E}, \cite{HM}, \cite{LM}). According to the Fulton-Hansen Connectedness Theorem \cite[Corollary~1]{FH}, $F_1(x,X)$ is connected by the assumption that $\dim F_1(x,W)>{\rm codim}_{\PP^N}\Lambda$. Furthermore, from the generality of $x \in X$, $F_1(x,X)$ is smooth. Hence the first part of our assertion holds. The second part follows from Proposition~\ref{fano}. 
\end{proof}

\subsection{Fano manifolds}

A {\it Fano manifold} means a projective manifold $X$ with ample anticanonical divisor $-K_X$. For a Fano manifold $X$, the {\it index} is defined as the greatest positive integer $r_X$ such that $-K_X=r_XH$ for some ample divisor $H$ on $X$, and the {\it pseudo-index} is defined as the minimum  $i_X$ of the intersection numbers of the anticanonical divisor with rational curves on $X$. Given a projective manifold $X$, we denote by $\rho_X$ the Picard number of $X$.  

\begin{them}[\cite{Wi}]\label{Wi} Let $F$ be a smooth Fano $n$-fold of index $r_F$ and pseudo-index $i_F$. 
\begin{enumerate}
\item If $i_F > \frac{n}{2}+1$, then $\rho_F=1$, and
\item if $r_F=\frac{n}{2}+1$ and $\rho_F \geq 2$, then $F \cong \PP^{r_F-1} \times \PP^{r_F-1}$.
\end{enumerate}
\end{them}

\begin{them}[\cite{KO}]\label{KO} Let $F$ be a smooth Fano $n$-fold of index $r_F$. Then $r_F \leq n+1$. Furthermore, if $r_F = n+1$, then $F$ is isomorphic to $\PP^n$. If $r_F = n$, then $F$ is isomorphic to $Q^n$.

\end{them}

\begin{them}[\cite{F1,F2}]\label{del} Let $X$ be a smooth Fano $n$-fold with index $r_X=n-1$ (i.e. a del Pezzo manifold) whose Picard group is generated by a very ample line bundle.  Then $X$ is isomorphic to one of the following: 
\begin{enumerate}
\item a cubic hypersurface, 
\item a complete intersection of two quadric hypersurfaces, or 
\item a linear section of the Grassmann variety $G(2,\CC^5) \subset \PP_{\ast}(\bigwedge^2\CC^5)$.
\end{enumerate}
\end{them}

\subsection{Extremal contractions}\label{extcont}

Let $X$ be a projective manifold and $NE(X)$ the cone of effective $1$-cycles on $X$. 
By the Contraction Theorem, given a $K_X$-negative extremal ray $R$ of the Kleiman-Mori cone $\overline{NE}(X)$, we obtain the contraction of the extremal ray $\varphi_R :X \rightarrow Y$. 
We say that $\varphi_R$ is {\it of fibering type} if $\dim X > \dim Y$.

Let $(X,H)$ be a pair consisting of a projective manifold $X$ of dimension $n$ and an ample line bundle $H$ on $X$, that is, a smooth polarized $n$-fold. %We say that $(X,L)$ is a {\it scroll} over a normal variety $Y$ of dimension $m$ if there exists a morphism with connected fibers $\varphi : X \rightarrow Y$ such that $K_X+(n-m+1)L = \varphi ^{\ast}A$ for some ample line bundle $A$ on $Y$. 
$(X,H)$ is called a {\it linear $\PP^d$-bundle} over a normal variety of dimension $m$ if there exists a surjective morphism $\varphi : X \rightarrow Y$ such that every fiber $(F,H|_F)$ is isomorphic to $(\PP^d,{\mathscr{O}}_{\PP^d}(1))$, where $d=n-m$. This is equivalent to say that 
there exists a surjective morphism $\varphi : X \rightarrow Y$ such that  $\E:=\varphi_{{\ast}}H$ is a locally free sheaf of rank $d+1$ and $(X,H)$ is isomorphic to $(\PP(\E),{\mathscr{O}}_{\PP(\E)}(1))$, where ${\mathscr{O}}_{\PP(\E)}(1)$ is the tautological line bundle on $\PP(\E)$ (see \cite[Section~3.2]{BS}).

%\begin{them}[{\cite[Proposition~14.1.3, Conjecture~14.1.10]{BSW},\cite[Theorem~1.7]{Ein}}]\label{lbdl} Let $X$ be a smooth projective $n$-fold and $L$ a very ample line bundle. Assume that $(X,L,\varphi )$ is a scroll, where $\varphi  :X \rightarrow Y$ is as in the definition of a scroll. Assume that 
%\begin{enumerate}
%\item $n \geq 2\dim Y+1$, or
%\item $3 \geq \dim Y$ and $n \geq 2\dim Y-1$.
%\end{enumerate}
%Then $\varphi $ is a linear projective bundle.
%\end{them}

\begin{them}[{\cite[Theorem~1.7]{Ein}}]\label{ein} Let $X \subset \PP^N$ be a smooth projective $n$-fold. Assume that there is an $f$-dimensional linear subspace $\Lambda$ in $X$ such that $N_{\Lambda/ X} \cong {\mathscr{O}}^{n-f}$ and $n \leq 2f-1$. Then $X \cong \PP_Y(\E)$, where $\E$ is a locally free sheaf of rank $f+1$ on a smooth projective $(n-f)$-fold $Y$.

\end{them}

\begin{pro}[{\cite[Lemma~2.12]{fujita}}]\label{fujita} Let $\varphi:X \rightarrow Y$ be a surjective morphism from a manifold to a normal projective variety and $H$ an ample line bundle on $X$. Assume that $(F,H|_F) \cong (\PP^d,{\mathscr{O}}_{\PP^d}(1))$ for a general fiber $F$ of $\varphi$ and every fiber of $\varphi$ is $d$-dimensional. Then $Y$ is smooth and $\varphi$ makes $(X,H)$ a linear $\PP^d$-bundle with $X \cong \PP(\varphi_{{\ast}}H)$.
\end{pro}

\begin{pro}[\cite{Wis}]\label{Wis} Let $X$ be a projective manifold. Let $\varphi_R$ be the contraction of an extremal ray $R \subset \overline{NE}(X)$, $E$ the exceptional locus of $\varphi_R$ and $F$ a component of a non-trivial fiber of $\varphi_R$. Then 
\begin{eqnarray}
\dim E \geq \dim X -1 +l(R)-\dim F, \nonumber
\end{eqnarray}
where $l(R):={\rm min} \{-K_X.C~|~C {\rm ~is~ a~ rational~ curve~ in~}  R\}$.
\end{pro}

The following theorem plays an important role in our proof.

\begin{them}[{\cite[Theorem~3.3]{NO}}]\label{NO} Let $X \subset \PP^N$ be a smooth projective $n$-fold with a covering family of lines $\K$. Denote by $[\K]$ the numerical class of a member of $\K$. Assume that $-K_X.\K \geq \frac{n+1}{2}$. Then $[\K]$ spans an extremal ray of $NE(X)$.% and $M:=K_X+mH$ is the pull-back of a divisor by the contraction of the extremal ray $\RR_{\geq 0}[\K]$. Furthermore,  
 %if $M$ is not nef, then there exists an extremal ray $R \subset NE(X)$ satisfying the following conditions:
%\begin{enumerate}
%\item $M.R<0$,
%\item $NE(X)=\langle [\K],R \rangle$, and
%\item every non-trivial fiber of the contraction of the extremal ray $R$ has dimension $m$ and $m=\frac{n+1}{2}$.
%\end{enumerate}
\end{them}

%\begin{proof} The first part of the assertion is \cite[Theorem~3.3]{NO}. It follows from \cite[Corollary~3.17]{KM} that $L$ is the pull-back of a divisor by the contraction of the extremal ray $\RR_{\geq 0}[\K]$. We see that the second part comes from their proof.
%\end{proof}

As a corollary of a generalization of Theorem~\ref{NO}, C.Novelli and G.Occhetta prove the following:

\begin{them}[{\cite[Corollary~4.4]{NO2}}]\label{NO2} Let $X \subset \PP^N$ be a projective manifold of dimension at most five and $\K$ a covering family of lines on $X$. Then $[\K]$ spans an extremal ray of $NE(X)$.
\end{them}

\begin{rem} \rm Theorem~\ref{NO} and Theorem~\ref{NO2} hold in a slightly more general situation \cite{NO}. We do not mention the original result precisely, because we do not need it. %Proposition~\ref{NO3} and Theorem~\ref{NO2} below also hold in a more general situation. 

\end{rem}

%\begin{pro}\label{NO3} Under the assumption of Theorem~\ref{NO}, if $L$ is not nef and $n \geq 3$, then $X$ is a Fano manifold and the contraction of the extremal ray $\varphi _R$ is a linear $\PP^m$-bundle. 
%\end{pro}

%\begin{proof} By Theorem~\ref{NO}, there exists an extremal ray $R \subset NE(X)$ satisfying $M.R<0$ and $NE(X)=\langle [\K],R \rangle$. Moreover every non-trivial fiber $F$ of the contraction of the extremal ray $\varphi _R$ has dimension $m=\frac{n+1}{2}$. Since $-M$ is non-negative on $NE(X)=\langle [\K],R \rangle$, it is nef. Therefore $-K_X=-M+mH$ is very ample. The length $l(R)$ is greater than $m$ because $M.R<0$. Applying Proposition~\ref{Wis}, we see that 
%\begin{eqnarray}
%\dim E \geq (\dim X -1)+ (l(R) - \dim F) > \dim X -1. \nonumber
%\end{eqnarray}
%Hence the contraction of the extremal ray $\varphi_R$ is of fibering type. 

%We also see that 
%\begin{eqnarray}\label{e} 
%-K_F=-K_X|_F=-M|_F+mH|_F. 
%\end{eqnarray}
%This implies that the pseudo-index $i_F$ is at least $m+1$. It follows from Theorem~\ref{Wi} that $\rho_F=1$. More strictly, we see that ${\rm Pic}(F) \cong \ZZ$ because $F$ is a Fano manifold. As is well known, the index $r_F$ is at most $\dim F+1$ from Theorem~\ref{KO}. By the relation~(\ref{e}), we obtain $r_F = m+1$ and ${\rm Pic}(F) \cong \ZZ[H|_F]$. Theorem~\ref{KO} again implies that $(F,{\mathscr{O}}_H(1)|_F)$ is isomorphic to $(\PP^m,{\mathscr{O}}_{\PP^m}(1))$. By Proposition~\ref{ein}, $\varphi _R$ is a linear projective bundle.
%\end{proof}

\section{Projective manifolds swept out by cubic hypersurfaces}  

In this section, we prove Theorem~\ref{hyp}. 

\begin{lem}\label{famcubic} Let $X \subset \PP^N$ be as in Theorem~\ref{hyp}. Then there exists a covering family of lines $\K$ and a smooth cubic hypersurface $S_x$ through a general point $x \in X$ such that any line lying in $S_x$ is a member of $\K$.
\end{lem}

\begin{proof} Let $P_1(t)$, $P_2(t)$ be the Hilbert polynomials of a line, an $s$-dimensional smooth cubic hypersurface, respectively and set $\PP(t):=(P_1(t),P_2(t))$. We denote the natural projections by 
\begin{eqnarray}
p_i: {\rm FH}_{\PP(t)}(X) \rightarrow {\rm Hilb}_{P_i(t)}(X),~{\rm where}~i=1,2. \nonumber 
\end{eqnarray} 

Let $\H_2$ be the open subscheme of ${\rm Hilb}_{P_2(t)}(X)$ parametrizing smooth subvarieties of $X$ with Hilbert polynomial $P_2(t)$. 
According to \cite{X}, a smooth projective $s$-fold with Hilbert polynomial $P_2(t)$ is a cubic hypersurface. 
Hence $\H_2$ parametrizes $s$-dimensional cubic hypersurfaces lying on $X$. 
 From our assumption, there is an irreducible component $\H_2^0$ of $\H_2$ such that $\overline{{\rm Locus}(\H_2^0)}=X$. Then $p_1(p_2^{-1}(\H_2^0))$ parametrizes lines lying on smooth cubic hypersurfaces in $\H_2^0$. Take an irreducible component $\K^0$ of $p_1(p_2^{-1}(\H_2^0))$ such that $\overline{{\rm Locus}(\K^0)}=X$. 
Through a general point $x$ on $X$, there is a  $\K^0$-line $l_x$ which is not contained in any irreducible component of $p_1(p_2^{-1}(\H_2^0))$ except $\K^0$. Furthermore there is also a smooth cubic hypersurface $S_x$ which satisfies with $[S_x]$ in $\H_2^0$ and $l_x \subset S_x$. Because $p_1(p_2^{-1}([S_x]))$ is the Hilbert scheme of lines on $S_x$, it is irreducible (see \cite[Theorem~1.10, Theorem~1.16, Proposition~1.19]{AK}). Therefore $p_1(p_2^{-1}([S_x]))$ is contained in an irreducible component of $p_1(p_2^{-1}(\H_2^0))$. Since $p_1(p_2^{-1}([S_x]))$ contains $[l_x]$, this implies that $p_1(p_2^{-1}([S_x]))$ is contained in $\K^0$. 
Thus we set $\K$ to be an irreducible component of $F_1(X)={\rm Hilb}_{P_1(t)}(X)$ containing $\K^0$.
\end{proof}

\begin{pro}[{\cite[Chap. I. Proposition~2.16]{Za}}]\label{Zak} Let $X \subset \PP^N$ be a nondegenerate smooth projective $n$-fold and $Y$ an $r$-dimensional subvariety of $X$ such that ${\rm codim}_{\langle Y \rangle}(Y)<{\rm codim}_{\PP^N}(X)=N-n$, where $\langle Y \rangle$ is the linear span of $Y$. Then $r \leq {\rm min} \{ n-1, [\frac{N-1}{2}] \}$. 
\end{pro}

\begin{proof}[Proof of Theorem~\ref{hyp}] We employ the notation as in Lemma~\ref{famcubic}. We denote $-K_X.\K$ by $m$ and a hyperplane section of $X \subset \PP^N$ by $H$. Set $L:=K_X+mH$. From a  fundamental deformation theory \cite[Theorem~4.3.5]{Ser}, we see that $\dim F_1(x, S_x) \geq s-3$.  
Since $F_1(x,S_x)$ is contained in $\K_x:=\{[l] \in \K| x \in l\}$ for a general point $x$ on $X$, we have  
\begin{eqnarray}\label{mi} 
m = \dim \K_x +2 \geq \dim F_1(x,S_x) +2 \geq s -1 \geq [\frac{n}{2}] +2 \geq \frac{n+3}{2}.  
\end{eqnarray}
As we noted in Section~\ref{l}, $\K_x$ is a disjoint union of smooth varieties of dimension $m-2$ and it is embedded into $\PP_{\ast}(T_xX)$ by the tangent map. From the inequality (\ref{mi}), we have $2(m-2) \geq n-1=\dim \PP_{\ast}(T_xX)$. Hence $\K_x$ is irreducible by the projective dimension theorem. 

By Theorem~\ref{NO}, $[\K]$ spans an extremal ray of $NE(X)$ and $L$ is the pull-back of a divisor by the extremal contraction $\varphi_{\K}:X \rightarrow Y$ associated to $\RR_{\geq 0}[\K]$.
For a general fiber $F$ of $\varphi_{\K}$, we have $K_F=K_X|_F=-mH|_F$. Hence the pseudo-index $i_F$ of $F$  satisfies $i_F \geq m \geq \frac{n+3}{2}$. By virtue of Theorem~\ref{Wi}, we conclude that $\rho_F=1$. 

We show that $r_F \geq f-1$, where $f:= \dim F$ and $r_F$ is the index of $F$. 
 We define $\K_F$ as the set of $\K$-lines which meet $F$, that is, $\K_F:=\{[l] \in \K|l\cap F\neq  \emptyset\}$. Since every $\K$-line is contracted to a point by $\varphi_{\K}$, we see that $\K_F=\{[l] \in \K| l \subset F\}$. From the construction, it follows that $\K_F$ is proper and ${\rm Locus}(\K_F)=F$. This implies that there is an irreducible component $\K_F^0$ of $\K_F$ such that ${\rm Locus}(\K_F^0)=F$. Therefore there exists a covering family ${\mathscr{L}}$ of lines on $F$ containing $\K_F^0$. 
Let $x$ be a general point on $F$ and set ${\mathscr{L}}_x:=\{[l] \in {\mathscr{L}}|x \in l\}$. Given any ${\mathscr{L}}_x$-line, it is free in $F$. This implies that the Hilbert scheme of lines $F_1(F)$ is smooth at $[l_x]$. In particular, there exists a unique irreducible component of $F_1(F)$ containing $[l_x]$. $\K_x$ is contained in ${\mathscr{L}}$, because it is irreducible and $\K_x \cap {\mathscr{L}}$ is not empty. So we have $\dim {\mathscr{L}}_x \geq \dim \K_x \geq \frac{n-1}{2}$. 
From Proposition~\ref{VMRT}, it follows that ${\mathscr{L}}_x \subset \PP_{\ast}(T_xF)$ is a nondegenerate projective manifold. Let $Z$ be an irreducible component of $F_1(x,S_x)$ with $\dim Z= \dim F_1(x,S_x)$. Remark that $\dim Z \geq s-3$ and $\dim {\mathscr{L}}_x=r_F-2$.
Now we apply Proposition~\ref{Zak} to $Z \subset {\mathscr{L}}_x \subset \PP_{\ast}(T_xF)$. It implies that 
\begin{eqnarray}\label{} 
2={\rm codim}_{\langle Z \rangle} Z \geq {{\rm codim}_{\langle {\mathscr{L}}_x \rangle}}{\mathscr{L}}_x= f +1 -r_F. \nonumber
\end{eqnarray}

According to Theorem~\ref{KO}, it follows that $r_F \leq f+1$ and if $r_X$ is  equal to $f +1$, respectively $f$, then $F$ is $\PP^f$, respectively a smooth quadric $Q^f$. In the case where $F$ is $\PP^f$, we see that $m=f+1$. %Hence $(X,H)$ is a scroll in the sense of Section~\ref{extcont}.  
We have 
\begin{eqnarray}\label{} 
f \geq s+1 \geq   [\frac{n}{2}]+4 \geq \frac{n+7}{2}.  \nonumber
\end{eqnarray}
It follows from Theorem~\ref{ein} that $(X,H)$ is a linear projective bundle.

We show that $F$ is not isomorphic to $Q^f$. Assume the contrary. 
Then $Q^f$ contains a smooth cubic hypersurface $S_x$ of dimension $s > [\frac{n}{2}]+2$. If $Q^f$ does not contain the linear span $\langle S_x \rangle$, an $s$-dimensional quadric $Q^f \cap \langle S_x \rangle$ contains a cubic $S_x$, a contradiction. Hence $Q^f$ contains a linear subspace $\langle S_x \rangle$ of dimension $s+1>[\frac{n}{2}]+3$. However an $f$-dimensional smooth quadric does not contain a linear subspace of dimension $> [\frac{f}{2}]$. So this also gives rise to a contradiction. As a consequence, $F$ is not isomorphic to $Q^f$.

In the case where $r_F=f-1$, $F$ is one of manifolds in Theorem~\ref{del}.  Since $F$ contains a smooth cubic hypersurface $S_x$ of dimension $s > [\frac{n}{2}]+2$, it follows from Proposition~\ref{Zak} that $F$ is neither a complete intersection of two quadric hypersurfaces nor a linear section of the Grassmann variety $G(2,\CC^5) \subset \PP_{\ast}(\bigwedge^2\CC^5)$. If $F$ is a cubic hypersurface, then $\varphi_{\K}:X \rightarrow Y$ is a cubic fibration.

Thus we have completed the proof of Theorem~\ref{hyp}.
\end{proof}

\section{Projective $5$-folds swept out by planes}

In this section, we prove the following:

\begin{them}\label{plane} Let $X \subset \PP^N$ be a smooth projective $5$-fold such that there is a plane $\PP^2 \subset X$ through a general point of $X$. Then $X$ is one of the following:
\begin{enumerate}
\item a linear $\PP^d$-bundle $\varphi :X \rightarrow Y$, where $d \geq 2$,
\item a smooth quadric $Q^5$, 
\item a complete intersection of two quadrics in $\PP^7$,
\item a hyperplane section of the Grassmann variety $G(2,\CC^5) \subset \PP_{\ast}(\bigwedge^2\CC^5)$, 
\item $X$ admits a contraction $\varphi :X \rightarrow C$ of an extremal ray to a smooth curve and, for the locally free sheaf $\E:=\varphi _{\ast}{\mathscr{O}}_X(1)$, $X$ embeds over $C$ into $\PP(\E)$ as a divisor of relative degree $2$, or
\item $X$ admits a contraction $\varphi  :X \rightarrow C$ of an extremal ray to a smooth curve whose general fiber is isomorphic to $\PP^2 \times \PP^2$.
\end{enumerate}

\end{them}

\subsection{Family of lines}

\begin{lem}\label{fam} Let $X \subset \PP^N$ be as in Theorem~\ref{plane}. Then there exists a covering family of lines $\K$ and a plane $P_x$ through a general point $x\in X$ such that any line lying in $P_x$ is a member of $\K$.
\end{lem} 

\begin{proof} Remark that the Hilbert scheme of lines on a plane is isomorphic to $G(2,\CC^3)$. In particular, it is irreducible. This lemma follows from the same argument as in Lemma~\ref{famcubic}.
\end{proof} 

\begin{nota}\label{not} \rm Let $X$, $\K$ and $P_x$ be as in Lemma~\ref{fam}. Let $\C_x$ be the variety of minimal rational tangents $\C_x$ with respect to $\K$ at a general point $x$ on $X$. We denote $-K_X.\K$ by $m$ and $K_X+mH$ by $L$, where $H$ is a hyperplane section of $X \subset \PP^N$. Throughout this section, we use these notation. 
\end{nota}

\begin{lem}\label{m} $\C_x$ contains a projective line $\PP^1$. In particular, $m \geq 3$.
\end{lem}

\begin{proof} From Lemma~\ref{fam}, it follows that $\C_x \subset \PP_{\ast}(T_xX)$ contains a line $\PP_{\ast}(T_xP_x)$. On the other hand, $m$ is equal to $\dim \C_x +2$. So our assertion holds. 
\end{proof}

\begin{lem}\label{contractions} The following holds.
\begin{enumerate}
\item $[\K]$ spans an extremal ray of $NE(X)$.
\item For the contraction $\varphi _\K: X \rightarrow Y$ of the extremal ray $\RR_{\geq 0}[\K]$, a general fiber $F$ of $\varphi _\K$ is a Fano manifold of index $r_F \geq m \geq 3$. 
\end{enumerate}
\end{lem}

\begin{proof} $\rm (i)$ By Lemma~\ref{m}, our setting satisfies the assumption of Theorem~\ref{NO}. Hence $[\K]$ spans an extremal ray.

$\rm (ii)$ Since $L$ is the pull-back of a divisor of $Y$, we have 
\begin{eqnarray}\label{} 
-K_F=-K_X|_F=mH|_F. \nonumber
\end{eqnarray} 
Hence our assertion holds.
\end{proof}

\begin{pro}\label{i3} We have $r_F \geq \dim F-1$. 
\end{pro}

\begin{proof} If $F$ does not coincide with $X$, then our assertion follows from Lemma~\ref{contractions} $\rm (ii)$. So we assume that $F=X$. Since $X \subset \PP^N$ is uniruled by lines and $\rho_X=1$, it is easy to see that $X$ is a Fano manifold whose Picard group is generated by a hyperplane section $H$. By virtue of Lemma~\ref{contractions}, we see that the index $r_X$ is equal to $m \geq 3$. Hence it is enough to prove that $r_X>3$. 

Assume that $r_X=3$. Then $X \subset \PP^N$ is a Mukai $5$-fold whose Picard group is generated by a very ample line bundle. We also see that the dimension $p$ of the variety of minimal rational tangents at a general point is $1$. From \cite{Muk}, $X \subset \PP^N$ is projectively equivalent to one of the following: 
\begin{enumerate}
\item a hypersurface of degree $4$,
\item a complete intersection of a quadric hypersurface and a cubic hypersurface,
\item a complete intersection of three quadric hypersurfaces,
\item a hyperplane section of ${\sum} \subset \PP^{10}$, where ${\sum} \subset \PP^{10}$ is a smooth section of the cone $\tilde{G}:={\rm Cone}(\{o\},G(2,\CC^5)) \subset \PP^{10}$ over the Grassmann variety $G(2,\CC^5) \subset \PP_{\ast}(\bigwedge^2 \CC^5)$ by a quadric hypersurface,  
\item a linear section of $S_4$, where $S_4$ is the spinor variety which is an irreducible component of the Fano variety of $4$-planes in $Q^{8}$,
\item a linear section of the Grassmann variety $G(2,\CC^6) \subset \PP_{\ast}(\bigwedge^2 \CC^6)$,
\item a hyperplane section of $LG(3,\CC^6)$, where $LG(3,\CC^6)$ is the Lagrangian Grassmann variety which is the variety of isotropic $3$-planes for a non-degenerate skew-symmetric bilinear form on $\CC^6$, or
\item the $G_2$-variety which is the variety of isotropic $5$-planes for a non-degenerate skew-symmetric $4$-linear form on $\CC^7$.

\end{enumerate}
By Lemma~\ref{m}, it is sufficient to show that $\C_x$ is an irreducible variety which is not a line.   

\begin{flushleft}
\underline{Cases (i)-(iii)} 
\end{flushleft}

Let $X \subset \PP^N$ be a smooth complete intersection of hypersurfaces of degree $d_1, \cdots, d_m$, where $p:=N-1- \sum d_i$ is positive. Then $X$ is covered by lines. 
For a general point $x$ on $X$, the Hilbert scheme $F_1(x,X)$ can be embedded into $\PP_{\ast}(T_xX)$ via the tangent map. Then the defining equations of $F_1(x,X) \subset \PP_{\ast}(T_x X)$ are given by derivatives at $x$ of the defining equations of $X$ (c.f.\cite[1.4.2]{Hw2}). So $F_1(x,X)$ is
a smooth complete intersection of dimension $p>0$. Since a smooth complete intersection of positive dimension is connected, we see that $F_1(x,X)$ is irreducible. According to Proposition~\ref{fano}, this implies that there exists a unique covering family of lines and $F_1(x,X) \subset \PP_{\ast}(T_xX)$ is its variety of minimal rational tangents $\C_x$. 

In our cases, if $X$ is a hypersurface of degree $4$, then $\C_x$ is a complete intersection of degrees $(4,3,2)$. If $X$ is a complete intersection of degrees $(3,2)$, then $\C_x$ is a complete intersection of degrees $(3,2,2)$. If $X$ is a complete intersection of degrees $(2,2,2)$, then so is $\C_x$. In particular, $\C_x$ is an irreducible variety which is not a line.

\begin{flushleft}
\underline{Case (iv)} 
\end{flushleft}

Let $X$ be a smooth hyperplane section ${\sum} \cap M \subset \PP^{10}$, where $M$ is a hyperplane of $\PP^{10}$. 

\begin{cl}\label{a} For ${\sum} \subset \PP^{10}$, there exists a unique covering family of lines $\K_{\sum}$ on ${\sum}$ and its variety of minimal rational tangents $\C_{\sum,x}$ at a general point $x$ is an intersection of a cone over a Segre $3$-fold $\PP^1 \times \PP^2$ and a $7$-dimensional quadric hypersurface. In particular, $\C_{\sum,x}$ is irreducible variety of degree $>1$. 
\end{cl}

\begin{proof}[Proof of Claim~\ref{a}] Let $\tilde{G}:={\rm Cone}(\{o\},G(2,\CC^5)) \subset \PP^{10}$ be the cone over the Grassmann variety $G(2,\CC^5) \subset \PP_{\ast}(\bigwedge^2 \CC^5)$ with vertex $o \in \PP^{10}$ and $\pi$ the associated projection $\PP^{10} \setminus \{o\} \rightarrow \PP^9$. For a point $x \in \tilde{G} \setminus \{o\}$, we denote by $\vec{xo} \in \PP_{\ast}(T_x\PP^{10})$ the point corresponding to the line passing through $x$ and $o$. Then we have a natural projection $\PP_{\ast}(T_x\PP^{10}) \setminus \{\vec{xo}\} \rightarrow \PP_{\ast}(T_{\pi(x)}\PP^9)$. As we mentioned before, $F_1(x, \tilde{G})$ can be embedded into $\PP_{\ast}(T_x\PP^{10})$ via the tangent map. Then $F_1(x, \tilde{G}) \subset \PP_{\ast}(T_x\PP^{10})$ is the cone with vertex $\vec{xo}$ over $F_1(\pi(x),G(2,\CC^5))$. 

Remark that $F_1(\pi(x),G(2,\CC^5)) \subset \PP_{\ast}(T_{\pi(x)}\PP^9)$ is projectively equivalent to the Segre $3$-fold $\PP^1 \times \PP^2 \subset \PP^5$ (see \cite{E}, \cite{HM}, \cite{LM}). On the other hand, for $x \in Q^9$, $F_1(x, Q^9)$ is a $7$-dimensional quadric hypersurface $Q^7 \subset \PP_{\ast}(T_x\PP^{10})$. Hence $F_1(x, {\sum})$ is the intersection of the cone of the Segre $3$-fold $\PP^1 \times \PP^2$ and $Q^7$. Thus we see that $F_1(x, {\sum})$ is connected. On the other hand, it follows from the generality of $x$ on ${\sum}$ that $F_1(x, {\sum})$ is smooth. Consequently, $F_1(x, {\sum})$ is irreducible. By virtue of Proposition~\ref{fano}, this implies that there exists a unique covering family of lines on $\sum$ and $F(x, {\sum})$ is its variety of minimal rational tangents. 
\end{proof}

Let $\K_X$ be a covering family of lines on $X$. From the uniqueness of a covering family of lines on ${\sum}$, it follows that $\K_X$ is contained in $\K_{\sum}$. By virtue of Proposition~\ref{nf}, $X \setminus Z_{\K_{\sum}}$ is a non-empty open subset of $X$. Denote by $\C_{X,x}$ the variety of minimal rational tangents with respect to $\K_X$ at $x$ on $X$. Fix a general point $x$ on $X$ (i.e., it is on $X \setminus (Z_{\K_{\sum}} \cup  Z_{\K_X})$). Then we obtain $\C_{X,x}=\C_{\sum,x} \cap \PP_{\ast}(T_xM)$. By the above Claim~\ref{a} and the same argument as in its proof, we see that $\C_{X,x}$ is irreducible variety of degree $>1$.

\begin{flushleft}
\underline{Cases (v)-(viii)} 
\end{flushleft}

In these cases, $X$ is a linear section of a rational homogeneous variety $W$. From Proposition~\ref{sechomog}, it follows that $\C_x$ is an irreducible variety of degree $>1$.

\end{proof}

\subsection{Case where $F$ coincides with $X$} 

Assume that a general fiber $F$ coincides with $X$.  As we mentioned in Proposition~\ref{i3} and its proof, $X$ is a Fano manifold whose Picard group is generated by a hyperplane section $H$ and the index $r_X$ is equal to $m > 3$.
From Theorem~\ref{KO}, it follows that $r_X \leq n+1=6$ and if $r_X$ is $6$, respectively $5$, then $X$ is isomorphic to $\PP^5$, respectively $Q^5$. If $X$ is isomorphic to $\PP^5$ or $Q^5$, then $X$ is swept out by planes. So we consider the case where $r_X=4$. In this case, $X$ is a del Pezzo $5$-fold whose Picard group is generated by a very ample line bundle. From Theorem~\ref{del}, $X \subset \PP^N$ is projectively equivalent to a cubic hypersurface, a complete intersection of two quadrics or a hyperplane section of the Grassmann variety $G(2,\CC^5) \subset \PP_{\ast}(\bigwedge^2\CC^5)$.

\begin{pro} A $5$-dimensional cubic hypersurface $X \subset \PP^6$ is not swept out by planes.
\end{pro}

\begin{proof} Since the dimension of the Hilbert scheme of planes in $X$ is $2$ by \cite[Proposition~3.4]{I}, $X$ is not swept out by planes.
\end{proof}

\begin{pro}[{\cite[Lemma~4.9]{R}}]\label{reid} A $(2n+1)$-dimensional smooth complete intersection $X \subset \PP^{2n+3}$ of two quadrics is swept out by $n$-dimensional linear subspaces. In particular, a $5$-dimensional smooth complete intersection $X \subset \PP^{7}$ of two quadrics is swept out by planes.
\end{pro}

\begin{pro} A hyperplane section of the Grassmann variety $G(2,\CC^5) \subset \PP_{\ast}(\bigwedge^2\CC^5)$ is swept out by planes. 
\end{pro}

\begin{proof} The Grassmann variety $G(2,\CC^5) \subset \PP_{\ast}(\bigwedge^2\CC^5)$ is covered by linear subspaces of dimension $3$ which are contained as subgrassmann varieties $G(1,\CC^4)$. Hence our assertion holds.
\end{proof}

\subsection{Case where $F$ does not coincides with $X$}\label{planetwo}

%\begin{lem}\label{fiber} Under the assumption of Lemma~\ref{contractions}, if $L$ is nef, $(F, {\mathscr{O}}_X(1)|_F)$ is isomorphic to one of the following:\begin{enumerate}
%\item $(\PP^d, {\mathscr{O}}_{\PP^d}(1))$, where $2 \leq d \leq 4$,
%\item $(Q^4, {\mathscr{O}}_{Q^4}(1))$,
%\item $\PP^2 \times \PP^2$.
%\end{enumerate}
%\end{lem}

%\begin{proof}  
%\end{proof}

\begin{pro}\label{P2} If a general fiber $F$ does not coincide with $X$, then $X$ is one of the following:
\begin{enumerate}
\item a linear projective bundle,
\item $X$ admits a contraction $\varphi  :X \rightarrow C$ of an extremal  ray to a smooth curve and, for the locally free sheaf $\E:=\varphi _{\ast}{\mathscr{O}}_X(1)$, $X$ embeds over $C$ into $\PP(\E)$ as a divisor of relative degree $2$, or 
\item $X$ admits a contraction $\varphi  :X \rightarrow C$ of an extremal  ray to a smooth curve whose general fiber is isomorphic to $\PP^2 \times \PP^2$.
\end{enumerate}
\end{pro}

\begin{proof} Since $\varphi _\K$ contracts planes and $F$ does not coincide with $X$, we see that $2 \leq \dim F \leq 4$. From Lemma~\ref{contractions}, the index $r_F$ of $F$ satisfies $r_F \geq m \geq 3$. By Theorem~\ref{Wi}, it holds that $\rho_F=1$ except the case $F$ is isomorphic to $\PP^2 \times \PP^2$. So we assume that $\rho_F=1$. Then $F$ is a linear subspace $\PP^d$, a quadric hypersurface $Q^d$ or a del Pezzo $4$-fold such that ${\rm Pic}(X) \cong \ZZ$ which is generated by a very ample line bundle. Since any smooth quadric of dimension $3$ does not contain a plane, $F$ is not isomorphic to $Q^3$. Furthermore, we see that $F$ is not a del Pezzo $4$-fold. In fact, if this is the case, then $F$ is isomorphic to one of manifolds as in Theorem~\ref{del}. However because $F$ is also swept out by planes, the variety of minimal rational tangents of $F$ at a general point contains a line. So by the same argument as in Proposition~\ref{i3}, we get a contradiction. As a consequence, $F$ is isomorphic to $(\PP^d, {\mathscr{O}}_{\PP^d}(1))$, $(Q^4, {\mathscr{O}}_{Q^4}(1))$ or $\PP^2 \times \PP^2$.

First, assume that $F$ is isomorphic to $\PP^d$, where $2 \leq d \leq 4$. If $F$ is isomorphic to $\PP^3$ or $\PP^4$, then it follows from Theorem~\ref{ein} that $\varphi _\K:X \rightarrow Y$ is a linear projective bundle. So we assume that $F$ is isomorphic to $\PP^2$. Then we show that $\varphi _\K:X \rightarrow Y$ is a linear $\PP^2$-bundle.

Since $L$ is the pull-back of a divisor on $Y$ by $\varphi_\K$ and a general fiber of $\varphi_\K$ is isomorphic to $\PP^2$, we see that $m=3$. By virtue of Proposition~\ref{fujita}, it is enough to show that every fiber of $\varphi_\K$ has dimension $2$. We assume the contrary and derive a contradiction. So assume that there exists a fiber $F_0$ of $\varphi_\K$ whose dimension is at least $3$. If $\dim F_0=4$, then we can take an irreducible component $E \subset F_0$ of dimension $4$. Slice $X$ and $E$ with $3$ general elements of $|H|$ and denote these by $S$ and $C$, respectively. We may assume that $S$ is a smooth surface and $C$ is a curve. The restriction of $\varphi_\K$ to $S$ is birational onto its image and contracts $C$. These imply that $C^2<0$. Thus we have $E.C=E|_S.C=C^2<0$. On the other hand, let $D$ be a curve contained in a fiber of $\varphi_\K$ which is not $F_0$. Then we have $E.D=0$. However, since $\varphi_\K$ is an extremal contraction, $C$ and $D$ are numerically proportional. This is a contradiction. Hence $\dim F_0=3$.

Slice $X$ and $F_0$ with $2$ general elements of $|H|$ and denote these by $X'$ and $F_0'$, respectively. Then we may assume that $X'$ is a smooth $3$-fold and $F_0'$ is a curve. Set $y:=\varphi_\K(F_0)$ and denote  by $\phi$ the restriction of $\varphi_\K$ to $X'$. Remark that $\phi$ is a birational morphism and its fiber has dimension at most $1$ near $y$. Since we have $-K_{X'}=(-K_X-2H)|_{X'}$ and $L$ is the pull-back of a divisor on $Y$ by $\varphi_\K$, we see that $-K_{X'}$ is $\phi$-ample. From the classification of extremal contractions from a smooth $3$-fold (see \cite{Cut} and remark that the argument as in \cite{Cut} holds even if the target space of an extremal contraction is affine), $\phi$ is a composition of divisorial contractions which contract a divisor to a smooth curve near $y$. 
We denote the composition by
\begin{eqnarray}
X_0':=X' \to X_1' \to X_2' \to \cdots \to X_s':=Y. \nonumber
\end{eqnarray}
Assume that $s \geq 2$. Denote by $l$ the fiber of $X_{s-1}' \to X_s':=Y$ over $y$ and by $\tilde{l} \subset X_{s-2}'$ its strict transform. For the blow-up $X_{s-2}' \to X_{s-1}'$, denote by $C \subset X_{s-1}'$ the blow-up center and by $E \subset X_{s-2}'$ the exceptional divisor. Since $l \cap C$ is not empty, we have $-K_{X_{s-2}'}.\tilde{l}=1-E.\tilde{l} \leq 0$. This is a contradiction. Hence we have $s=1$. So we see that $\phi$ is the blow-up along a smooth curve near $y$. This implies that $F_0'$ is isomorphic to $\PP^1$ and $N_{F_0'/X'} \cong {\mathscr{O}}_{\PP^1} \oplus {\mathscr{O}}_{\PP^1}(-1)$. Here, by the same argument as in \cite[Proposition~3.2.3]{BSW}, we get a contradiction. As a consequence, if a general fiber of $\varphi _\K$ is isomorphic to $\PP^2$, then $\varphi _\K:X \rightarrow Y$ is a linear $\PP^2$-bundle.

%Let $T$ be a general hyperplane section of $Y$ through $y$ and $S$ its pull-back by $\phi$. Then we may assume that $S$ is smooth near the fiber over $y$. Furthermore, since $\phi|_S: S \rightarrow T$ is a birational morphism and $-K_S$ is $\phi|_S$-ample, $\phi|_S$ contracts a unique $(-1)$-curve. Hence there exists a unique $3$-dimensional component $D$ of $F_0$ and $C:=D \cap X'$ is just the fiber of $\phi$ over $y$. 

Second, assume that $F$ is isomorphic to $Q^4$. Then $C:=Y$ is a smooth curve and $\varphi _\K$ is flat. By the Grauert's theorem \cite[III, Corollary~12.9]{H}, $\E:={\varphi _\K}_{\ast}{\mathscr{O}}_X(1)$ is locally free of rank $6$. Since ${\mathscr{O}}_X(1)$ is very ample, ${\varphi _\K}^{\ast}{\varphi _\K}_{\ast}{\mathscr{O}}_X(1) \rightarrow {\mathscr{O}}_X(1)$ is surjective. This yields the embedding of $X$ into $\PP(\E)$ over $C$. Finally, if $F$ is isomorphic to $\PP^2 \times \PP^2$, then there is nothing to prove. 
\end{proof}

%\begin{pro} Under the assumption of Lemma~\ref{contractions}, if $L$ is not nef, $X$ is isomorphic to $\PP^2 \times \PP^3$. In particular, in this case, $X$ is a linear $\PP^2$-bundle.
%\end{pro}

%\begin{proof} By Theorem~\ref{NO}, there exists an extremal ray $R \subset NE(X)$ satisfying $L.R<0$ and $NE(X)=<[\K],R>$. From Proposition~\ref{NO3}, $\varphi _R :X \rightarrow Z$ is a linear $\PP^3$-bundle. On the other hand, every fiber of $\varphi _{\K} :X \rightarrow Y$ has dimension $2$, since two rays $\RR_{\geq 0}[\K]$, $R$ are different. In particular, a general fiber of $\varphi _\K$ is a plane $\PP^2$. So Proposition~\ref{fujita} implies that $\varphi _{\K} :X \rightarrow Y$ is a linear $\PP^2$-bundle. 
%Then by a theorem of Lazarsfeld \cite[Theorem~4.1]{Laz}, we see that $Z$ (respectively $Y$) is isomorphic to $\PP^2$ (respectively $\PP^3$). So $\varphi _\K \times \varphi _R : X \rightarrow Y \times Z$ is bijective. Hence it is an isomorphism by Zariski's Main Theorem. \end{proof}

\begin{rem}\label{takagi} \rm Professor Hiromichi Takagi taught the author the above proof such that $\varphi _\K:X \rightarrow Y$ is a linear $\PP^2$-bundle if  
 $F$ is isomorphic to $\PP^2$. 
\end{rem}

We have thus completed the proof of Theorem~\ref{plane}.

\begin{rem}\label{NOQ} \rm A $(2n+1)$-dimensional smooth complete intersection of two quadrics $X$ gives a negative answer to the following question raised by Novelli-Occhetta \cite[Question~6.2]{NO}: 

{\it Let $Y \subset \PP^N$ be a smooth variety of dimension $2n+1$ such that through every point of $Y$ there is a linear subspace of dimension $n$. ... Is it true that there exists an $n$-dimensional linear subspace $\Lambda$ with nef normal bundle?}

Actually, $X$ is swept out by $n$-dimensional linear subspaces (see Proposition~\ref{reid}). However there does not exist an $n$-dimensional subspace with nef normal bundle in $X$ except for the case where $n=1$ (See \cite[Theorem~1.1, Theorem~7.1]{NO}).

\end{rem}

\section{Projective manifolds swept out by copies of the Segre $3$-fold $\PP^1 \times \PP^2 \subset \PP^5$}

In this section, we show Theorem~\ref{Segre3}. 

Let $X \subset \PP^N$ be as in Theorem~\ref{Segre3}. Then $X \subset \PP^N$ is also swept out by planes. If $n=3$, there is nothing to prove. Suppose that $n=4$. By virtue of \cite[Main Theorem]{S}, $X$ is a linear projective bundle or a quadric hypersurface $Q^4$. However a $4$-dimensional smooth quadric hypersurface does not contain a copy of the Segre $3$-fold $\PP^1 \times \PP^2$. Hence Theorem~\ref{Segre3} holds if $n < 5$.  Therefore we assume that $n=5$ from now on. 

\subsection{Families of lines}

The Segre $3$-fold $S=\PP^1 \times \PP^2$ contains two families of lines. We say that a line on $S$ is {\it of the first type} (resp. {\it{of the second type}}) if it is contracted to a point by the second projection $S=\PP^1 \times \PP^2 \rightarrow \PP^2$ (resp. the first projection $S=\PP^1 \times \PP^2 \rightarrow \PP^1$). 

Let $P_1(t)$, $P_2(t)$, $P_3(t)$ be the Hilbert polynomials of a line, a plane, the Segre $3$-fold $\PP^1 \times \PP^2$, respectively. Moreover, set $\PP_{123}(t):=(P_1(t), P_2(t), P_3(t))$ and $\PP_{13}(t):=(P_1(t), P_3(t))$. Then we have the following commutative diagram: 

\[\xymatrix{
& {\rm FH}_{\PP_{123}(t)}(X) \ar[dl]_{\overline{p_1}} \ar[d] \ar[dr]^{\overline{p_3}} & \\
{\rm F}_1(X) & {\rm FH}_{\PP_{13}(t)}(X) \ar[l]_{p_1} \ar[r]^{p_3} & {\rm Hilb}_{P_3(t)}(X) \\
} \]

In this subsection, we prove the following:  

\begin{pro}\label{summing} Let $X \subset \PP^N$ be a smooth projective $5$-fold such that there is a copy of the Segre $3$-fold $\PP^1 \times \PP^2$ through a general point of $X$ in $X$. Then $X$ admits two covering families of lines $\K_1$ and $\V_1$ which define extremal contractions $\varphi_{\K_1}$ and $\varphi_{\V_1}$ respectively. Furthermore, through a general point of $X$ there is a copy  $S_x \subset X$ of the Segre $3$-fold $\PP^1 \times \PP^2$ such that $\varphi_{\K_1}$ (resp. $\varphi_{\V_1}$) contracts lines in $S_x$ of the second type (resp. of the first type).
\end{pro}

\begin{lem} Let $\H_3$ be the open subscheme of ${\rm Hilb}_{P_3(t)}(X)$ parametrizing smooth subvarieties of $X$ with Hilbert polynomial $P_3(t)$. 
 Then $\H_3$ parametrizes copies of the Segre $3$-fold $\PP^1 \times \PP^2$ on $X$.
\end{lem}

\begin{proof} According to \cite{X}, a smooth cubic $3$-fold is projectively equivalent to a cubic hypersurface or the Segre $3$-fold $\PP^1 \times \PP^2$. The Hilbert polynomial of a cubic hypersurface does not coincide with $P_3(t)$. Hence our assertion holds.
\end{proof}

Because $X$ is swept out by copies of the Segre $3$-fold $\PP^1 \times \PP^2$, we see that $\overline{{\rm Locus}(\H_3)}=X$. 
Let $\H_3^0$ be an irreducible component of $\H_3$ such that $\overline{{\rm Locus}(\H_3^0)}=X$.

\begin{lem}\label{k1} There exists a covering family of lines $\K_1 \subset F_1(X)$ such that 
,through a general point $x$ on $X$, there exists a  $\K_1$-line $l$ of the second type on some $S$ with $[S] \in \H_3^0$. Furthermore $[\K_1]$ spans an extremal ray of $NE(X)$.
\end{lem} 

\begin{proof} Remark that any plane contained in a copy $S$ of the Segre $3$-fold $\PP^1 \times \PP^2$ is contracted to a point by the first projection $S=\PP^1 \times \PP^2 \rightarrow \PP^1$. Hence $\bar{p_1}(\bar{p_3}^{-1}(\H_3^0))$ parametrizes lines of the second type in copies of the Segre $3$-fold $\PP^1 \times \PP^2$ contained in $\H_3^0$. Since ${\rm Locus}(\bar{p_1}(\bar{p_3}^{-1}(\H_3^0)))={\rm Locus}(\H_3^0)$, we can take an irreducible component $\K_1^0$ of $\bar{p_1}(\bar{p_3}^{-1}(\H_3^0))$ such that $\overline{{\rm Locus}(\K_1^0)}=X$. Thus set $\K_1$ as an irreducible component of $F_1(X)$ containing $\K_1^0$. 

By the same argument as in Lemma~\ref{fam} and Lemma~\ref{m}, we see that $-K_X.\K_1 \geq 3$. Hence it follows from Theorem~\ref{NO} that $[\K_1]$ spans an extremal ray of $NE(X)$. 
\end{proof}

Denote by $\varphi_{\K_1}$ the contraction of the extremal ray $\RR_{\geq 0}[\K_1]$. 
From the construction of the family $\K_1$, $\varphi_{\K_1}$ satisfies the following property: 

\begin{lem}\label{planecontract}
Through a general point $x$ on $X$, there exists a copy $S$ of the Segre $3$-fold $\PP^1 \times \PP^2$ such that $[S] \in \H_3^0$ and every plane on $S$ is contracted by $\varphi_{\K_1}$.
\end{lem}

\begin{proof} Note that all the curves contained in planes on a copy $S$ of the Segre $3$-fold $\PP^1 \times \PP^2$ are numerically proportional to each other and a contraction of an extremal ray contracts all curves which are numerically proportional to a contracted curve. From Lemma~\ref{k1}, through a general point $x$ on $X$, there exists a  $\K_1$-line $l$ of the second type on some $S$ with $[S] \in \H_3^0$. Because $l$ is contracted by $\varphi_{\K_1}$, every plane on $S$ is contracted by $\varphi_{\K_1}$.
\end{proof}

\begin{lem} There exists a covering family of lines $\V_1 \subset F_1(X)$ such that, 
through a general point $x$ on $X$, there exists a $\V_1$-line $l$ of the first type on some $S$ with $[S] \in \H_3^0$, and $S$ contains a $\K_1$-line of the second type. Furthermore $[\V_1]$ spans an extremal ray of $NE(X)$ and every line on $S$ in $\H_3^0$ of the first type is contracted by $\varphi_{\V_1}$, where $\varphi_{\V_1}$ is the contraction of the extremal ray $\RR_{\geq 0}[\V_1]$.  

\end{lem} 

\begin{proof} We use the notation as in the proof of Lemma~\ref{k1}. Set $\M:=\bar{p_3}(\bar{p_1}^{-1}(\K_1))\cap \H_3^0$. Given any $S$ which is a member of $\M$, there exists a $\K_1$-line $m$ of the second type in $S$. From the definition of $\K_1^0$, we see that ${\rm Locus}(\K_1^0) \subset {\rm Locus}(\M)$ and $\overline{{\rm Locus}(\K_1^0)}=X$. It implies that $\overline{{\rm Locus}(\M)}=X$. So there exists an irreducible component $\M^0$ of $\M$ such that $\overline{{\rm Locus}(\M^0)}=X$. Then ${p_1}({p_3}^{-1}(\M^0))$ parametrizes lines in copies of the Segre $3$-fold $\PP^1 \times \PP^2$ contained in $\M^0$. On the other hand, $\bar{p_1}(\bar{p_3}^{-1}(\M^0))$ parametrizes lines of the second type in copies of the Segre $3$-fold $\PP^1 \times \PP^2$ contained in $\M^0$. Furthermore $\bar{p_1}(\bar{p_3}^{-1}(\M^0))$ is a proper closed subset of ${p_1}({p_3}^{-1}(\M^0))$ and ${\rm Locus}(\M^0)={\rm Locus}({p_1}({p_3}^{-1}(\M^0)) \setminus \bar{p_1}(\bar{p_3}^{-1}(\M^0))$. So we have an irreducible component $\V_1^0$ of ${p_1}({p_3}^{-1}(\M^0)) \setminus \bar{p_1}(\bar{p_3}^{-1}(\M^0))$ such that $\overline{{\rm Locus}(\V_1^0)}=X$. Thus set $\V_1$ as an irreducible component of $F_1(X)$ containing $\V_1^0$. The second assertion is derived from Theorem~\ref{NO2} and the same argument as in Lemma~\ref{planecontract}.  
\end{proof}

Hence Proposition~\ref{summing} holds.

We may consider $\varphi_{\K_1}$ as the contraction of an extremal ray $\varphi_{\K}$ appeared in Section~\ref{planetwo}. Now $X$ is isomorphic to one of manifolds as in Theorem~\ref{plane}. We study $X$ according to each case.

\subsection{Cases $\rm (ii)-(iv)$}

In this subsection, we study the case where $X$ is isomorphic to one of manifolds as in $\rm (ii)-(iv)$ of Theorem~\ref{plane}. 
 
\begin{pro}\label{sch} Every smooth hyperplane section of the Grassmann variety $G(2,\CC^5) \subset \PP_{\ast}(\bigwedge^2\CC^5)$ is swept out by copies of the Segre $3$-fold $\PP^1 \times \PP^2$.
\end{pro}

\begin{proof} Let $\Omega(1,4) \subset G(1, \PP^4)$ be a Schubert variety parametrizing the lines in $\PP^4$ meeting a given line. Because a general hyperplane section of $\Omega(1,4)$ is a copy of the Segre $3$-fold $\PP^1 \times \PP^2$, our assertion holds. 
\end{proof}

We use a corollary of Zak's theorem on tangencies \cite[Corollary~1.8]{Za}:

\begin{them}[{\cite[Proposition~4.2.8, Corollary~4.2.10]{FOV}}]\label{zaktang} Let $X \subset \PP^N$ be a nondegenerate subvariety and $\Lambda \subset \PP^N$ a linear subspace of codimension $r$. Assume that 
\begin{eqnarray}
N-\dim X + \dim (\Lambda \cap {\rm Sing}(X)) +r-1 < \dim X -r, \nonumber
\end{eqnarray}
where ${\rm Sing}(X)$ stands for the singular locus of $X$. 

Then $\dim X \cap \Lambda =\dim X -r$ and 
\begin{eqnarray}
\dim {\rm Sing}(X \cap \Lambda) \leq  N-\dim X + \dim (\Lambda \cap {\rm Sing}(X)) +r-1. \nonumber
\end{eqnarray}

Furthermore, if $X$ is locally Cohen-Macaulay and 
\begin{eqnarray}
2 \dim X - N- \dim (\Lambda \cap {\rm Sing}(X)) -2r \geq 0, \nonumber
\end{eqnarray}

then $X \cap \Lambda$ is reduced.
\end{them}

\begin{pro} A $5$-dimensional smooth quadric hypersurface $Q^5 \subset \PP^6$ does not contain a copy of the Segre $3$-fold $\PP^1 \times \PP^2$.
\end{pro}

\begin{proof} Assume that a copy $S$ of the Segre $3$-fold $\PP^1 \times \PP^2$ is contained in $Q^5 \subset \PP^6$. Let $H$ be the linear span of 
$S$ which is a hyperplane of $\PP^6$. Denote by $Y$ the hyperplane section of $Q^5$ by $H$. Then we see that $Y$ is a quadric of dimension $4$. From Theorem~\ref{zaktang}, it follows that the dimension of the singular locus of $Y$ is at most $0$. Therefore $Y$ is smooth or a cone over a smooth quadric hypersurface of dimension $3$. If $Y$ is smooth, then $S$ is a member of the linear system corresponding to the hyperplane section of $Y$. This gives a contradiction. So we may assume that $Y$ is a cone with vertex $o \in \PP^5$. We denote by $\pi: \PP^5 \cdots \rightarrow \PP^4$ the projection from a point $o$. Then $\PP^1 \times \PP^2 \setminus \{o\}$ contains a plane $\PP^2$ and the image $\pi(\PP^2)$ is also of degree $1$. However it contradicts the fact that any $3$-dimensional smooth quadric does not contain a plane.
\end{proof} 

\begin{pro} A $5$-dimensional smooth complete intersection $X$ of two quadrics does not contain a copy of the Segre $3$-fold $\PP^1 \times \PP^2$. 
\end{pro}

\begin{proof} Suppose that a copy $S$ of the Segre $3$-fold $\PP^1 \times \PP^2$ is contained in $X \subset \PP^7$. Denote by $H$ the linear span of $S$. According to Theorem~\ref{zaktang}, we see that $\dim (X \cap H)=3$ and $X \cap H$ is reduced. Since $X \cap H$ contains $S$, we obtain a $3$-dimensional linear subspace $L$ such that $X \cap H= S \cup L$. However this contradicts the fact that $X$ does not contain a $3$-dimensional linear subspace (see \cite[Corollary~2.4]{R}).
\end{proof}

\subsection{Cases $\rm (i), (v)$ and $\rm (vi)$}
If $\varphi_{\K_1}$ is a linear $\PP^d$-bundle with $d \geq 2$, there is nothing to prove. So we deal with the cases $\rm (v)$ and $\rm (vi)$.

\begin{pro}\label{fb}  Assume that $\varphi_{\K_1}:X \rightarrow Y$ is $\varphi$ as in $\rm (v)$ of Theorem~\ref{plane}. Then $X$ is $\PP^1 \times Q^4$.
\end{pro} 

\begin{proof} 
Because the Segre $3$-fold $\PP^1 \times \PP^2 \subset \PP^5$ is not contained in any smooth quadric $4$-fold, we see that $\RR_{\geq 0}[\K_1] \neq \RR_{\geq 0}[\V_1]$. 
Furthermore, the dimension of every fiber of $\varphi_{\K_1}$ is $4$.
Hence every fiber of $\varphi _{\V_1}$ is of dimension $1$ and then we see that a general fiber of $\varphi _{\V_1}$ is a $\V_1$-line. 
Hence it follows from Proposition~\ref{fujita} that $\varphi _{\V_1}: X \rightarrow Z$ is a linear $\PP^1$-bundle over a smooth variety $Z$. 

Here we show that $Z$ is isomorphic to $\PP^4$ or $Q^4$. 
Let $F$ be a general fiber of $\varphi _{\K_1}$, which is a smooth quadric of dimension $4$. Since $F$ is not contracted to a point by $\varphi _{\V_1}$, the dimension of $\varphi _{\V_1}(F)$ is $4$. It implies that $\varphi _{\V_1}: F \rightarrow Z$ is a finite surjective morphism. According to \cite[Proposition~8]{PS}, $Z$ is isomorphic to $\PP^4$ or $Q^4$. 

Since $\rho_X=2$, the Kleiman-Mori cone $NE(X)$ is generated by $\RR_{\geq 0}[\K_1]$ and $\RR_{\geq 0}[\V_1]$. This implies that $-K_X$ is ample. Hence $X$ is the projectivization of a rank two Fano bundle $\E$ on $\PP^4$ or $Q^4$. It follows that $X$ is isomorphic to $\PP^1 \times Q^4$ from the classification result of rank two Fano bundles in \cite[Main Theorem~2.4]{APW}, since $X$ has a contraction of an extremal ray onto a curve. Consequently, $X$ is projectively equivalent to $\PP^1 \times Q^4 \subset \PP^{11}$. 

\end{proof}

\begin{pro}\label{segre4} Assume that $\varphi_{\K_1}:X \rightarrow Y$ is $\varphi$ as in $\rm (vi)$ of Theorem~\ref{plane}. Then $\varphi _{\K_1} :X \rightarrow C$ is a contraction of an extremal ray to a smooth curve whose general fiber is projectively equivalent to the Segre $4$-fold $\PP^2 \times \PP^2 \subset \PP^8$.
\end{pro}

\begin{proof} Since a general fiber contains a copy of the Segre $3$-fold $\PP^1 \times \PP^2$, it is easy to see that a general fiber is projectively equivalent to the Segre $4$-fold $\PP^2 \times \PP^2 \subset \PP^8$.
\end{proof}

Hence we have shown Theorem~\ref{Segre3}.

\begin{rem}\rm  An example of Theorem~\ref{Segre3}~{\rm (iv)} is given in \cite[Example~3.6]{NO}.
\end{rem}

\end{document}